\documentclass[]{amsart}
\usepackage{amssymb,amsmath, graphicx}
\usepackage{mathdots}
\usepackage{mathrsfs}
\usepackage{wrapfig}
\usepackage{enumerate}
\usepackage{url}
\usepackage[utf8]{inputenc}
\usepackage{diagrams}

\def\resto#1#2{{
#1\hskip 0.4ex\vline_{\hskip 0.2ex\raisebox{-0,2ex}
{{${\scriptstyle #2}$}}}}}

\def\union{\mathop{\bigcup}}

\def\textmap#1{\mathop{\vbox{\ialign{
                                  ##\crcr
      ${\scriptstyle\hfil\;\;#1\;\;\hfil}$\crcr
      \noalign{\kern 1pt\nointerlineskip}
      \rightarrowfill\crcr}}\;}}
\def\bigtextmap#1{\mathop{\vbox{\ialign{
                                  ##\crcr
      ${\hfil\;\;#1\;\;\hfil}$\crcr
      \noalign{\kern 1pt\nointerlineskip}
      \rightarrowfill\crcr}}\;}}

\def\textlmap#1{\mathop{\vbox{\ialign{
                                  ##\crcr
      ${\scriptstyle\hfil\;\;#1\;\;\hfil}$\crcr
      \noalign{\kern-1pt\nointerlineskip}
      \leftarrowfill\crcr}}\;}}

\def\N{{\mathbb N}}

\def\R{{\mathbb R}}

\def\jg{{\mathfrak j}}
\def\kg{{\mathfrak k}}
\def\lg{{\mathfrak l}}

\def\pg{{\mathfrak p}}

\theoremstyle{remark}

\theoremstyle{plain}

\newtheorem{sz}{Satz}[section]
\newtheorem{thry}[sz]{Theorem}
\newtheorem{pr}[sz]{Proposition}

\newtheorem{co}[sz]{Corollary}
\newtheorem{dt}[sz]{Definition}
\newtheorem{lm}[sz]{Lemma}

\theoremstyle{remark}
\newtheorem{re}{Remark}

\theoremstyle{plain}

\def\End{\mathrm {End}}
\def\Aut{\mathrm {Aut}}

\def\U{\mathrm{U}}
\def\SU{\mathrm {SU}}

\def\PU{\mathrm {PU}}

\def\Iso{\mathrm {Iso}}

\def\Hom{\mathrm{Hom}}

\def\id{ \mathrm{id}}

\def\ad{\mathrm {ad}}
\def\U2{\mathrm{U(2)}}
\def\niq{=\kern-.18cm /\kern.08cm}

\def\ad{\mathrm{ad}}

\newcommand{\cal}{\mathcal}

\begin{document}

\title[Locally homogeneous triples]{Locally homogeneous triples. Extension theorems for parallel sections and parallel bundle isomorphisms}
\author{Arash Bazdar}
\address{Aix Marseille Université, CNRS, Centrale Marseille, I2M, UMR 7373, 13453 Marseille, France}
\email{arash.bazdar@univ-amu.fr}
\thanks{ I thank to my PhD advisor Andrei Teleman for suggesting me this interesting research topic, and for guiding my work on the subject. }

\begin{abstract}
Let $M$ be a differentiable manifold and $K$ a Lie group. A locally
homogeneous triple with structure group $K$ on $M$ is a triple $(g,
P\stackrel{p}{\to} M,A)$, where $p:P\to M$ is a principal $K$-bundle on $M$,
$g$ is Riemannian metric on $M$, and $A$ is connection on $P$ such that the
following locally homogeneity condition is satisfied: for every two points $x$,
$x'\in M$ there exists an isometry $\varphi:U\to U'$ between open neighbourhoods
$U\ni x$, $U'\ni x'$ with $\varphi(x)=x'$, and a $\varphi$-covering bundle
isomorphism $\Phi:P_U\to P_{U'}$ such that $\Phi^*(A_{U'})=A_U$. If
$(g,P\stackrel{p}{\to} M,A)$ is a locally homogeneous triple on $M$, one can
endow the total space $P$ with a locally homogeneous Riemannian metric such
that $p$ becomes a Riemannian submersion and $K$ acts by isometries. Therefore
the classification of locally homogeneous triples on a given manifold $M$ is an
important problem: it gives an interesting class of geometric manifolds which
are fibre bundles over $M$.

In this article we will prove a classification theorem for locally
homogeneous triples. We will use this result in a future article in order to
describe explicitly moduli spaces of locally homogeneous triples on Riemann
surfaces.	
\end{abstract}

\maketitle
{\ }\vspace{-5mm}
\section{Introduction}\label{introd}

\subsection{The problem} \label{TheProblem} 
Let $M$ be a smooth manifold and $K$ a Lie group. A locally homogeneous triple with structure group $K$ on $M$ is a triple $(g, P\textmap{p} M,A)$, where $p:P\to M$ is a principal $K$-bundle on $M$, $g$ is Riemannian metric on $M$, and $A$ is connection on $P$ such that the following  locally homogeneity condition is satisfied: for every two points $x$, $x'\in M$ there exists an isometry $\varphi:U\to U'$ between open neighbourhoods $U\ni x$, $U'\ni x'$ with $\varphi(x)=x'$,  and a $\varphi$-covering bundle isomorphism $\Phi:P_U\to P_{U'}$ such that $\Phi^*(A_{U'})=A_U$. In these formulae, for an open set $U\subset M$, we  use the subscript $_U$ to denote the restriction of the indicated objects to $U$.

If $(g,P\textmap{p} M,A)$ is a locally homogeneous triple on $M$, one can endow the total space $P$ with a locally homogeneous Riemannian metric such that $p$ becomes a Riemannian submersion and $K$ acts by isometries. Therefore the    classification of locally homogeneous triples on a given manifold $M$ is an important problem: it gives an interesting class of geometric manifolds which are fibre bundles  over $M$.

In this article we will prove  a classification theorem for locally homogeneous triples. We will use this result in a future article in order to describe explicitly the moduli spaces of locally homogeneous triples with structure group $K=\SU(2)$ and $K=\PU(2)$ on a Riemann surface \cite{BaTe}.

We use the following general natural idea: reduce the classification of a class of locally homogeneous objects on $M$ to the classification of a class of (globally) homogeneous  objects on the universal cover $\tilde M$.  Suppose for instance that $M$ is compact, and $g$ is a locally homogeneous  Riemannian metric on $M$. Then the induced metric $\tilde g$ on $\tilde M$ will be locally homogeneous and complete, hence homogeneous by a well-known theorem of Singer \cite{Si}.  Therefore, using the correspondence $g\mapsto\tilde g$, the classification of locally homogeneous  Riemannian metrics  $g$ on $M$ reduces to the classification of homogeneous metrics $\tilde g$ on $\tilde M$ for which the isometry group $\Iso(\tilde M,\tilde g)$  contains the covering transformation group of the universal cover $\tilde M\to M$.
We will prove a similar result for locally homogeneous triples.\\

\subsection{Locally homogeneous triples and homogeneous connections}
\label{LocHom-intro}

%

Let $(N,h)$ be a Riemannian manifold, and let $G\subset \Iso(N,h)$ be a connected, closed subgroup of the group $\Iso(N,h)$ of isometries of $(N,h)$. Let $q:Q\to N$ be a principal $K$-bundle on $N$. The group of $G$-covering bundle isomorphisms of $Q$ is defined by
$${\cal G}_G(Q):=\{(\Phi,\varphi)|\ \varphi\in G,\ \Phi:Q\to Q\hbox{ is a $\varphi$-covering bundle isomorphism}\}.$$  
The group ${\cal G}_G(Q)$ has a natural   topology (induced by the weak ${\cal C}^\infty$-topology \cite[section 2.1]{Hi}), and fits in the short exact sequence
$$\{1\}\to {\cal G}(Q)\to {\cal G}_G(Q)\textmap{\pg} G\to \{1\},
$$
where ${\cal G}(Q)$ is the gauge group of $Q$ \cite{DK}, \cite{Te}, i.e. the group of $\id$-covering bundle automorphisms of $Q$.  Let now $\Gamma\subset G$ be a subgroup of $G$ acting properly discontinuously on $N$. The quotient $M:=N/\Gamma$  comes with a natural Riemannian metric, such that the canonical projection $\pi:N\to M$ becomes a locally isometric covering  projection.   Denote by $g$ this Riemannian metric on $M$ induced by $h$ via $\pi$. 

Suppose  now the group epimorphism $\pg^{-1}(\Gamma)\to \Gamma$ has a right inverse, i.e. that  there exists a group morphism $\jg:\Gamma\to {\cal G}_G(Q)$ such that $\pg\circ \jg=\iota_\Gamma$, where $\iota_\Gamma:\Gamma\to G$ is the inclusion monomorphism. If this is the case, we will obtain the commutative diagram
\begin{equation}
\begin{diagram}[s=7mm]
\{1\}&\rTo  &{\cal G}(Q) &\rTo  &{\cal G}_G(Q) &\rTo^{\pg} & G&\rTo & \{1\}	\\
&&&&&\luTo_{\jg}&\uInto_{\iota_\Gamma}  && \\
&&&&&& \Gamma &&
\end{diagram}.
\end{equation}
The group $\Gamma$ acts (via $\jg$) on $Q$ by bundle isomorphisms, and the quotient $P:=Q/\Gamma$ will be a principal $K$-bundle on the quotient manifold $M$.\\

Let now $B$ be a connection on $Q$ satisfying the following invariance condition: \\
\\
{(C$_G$) {\it Any element $\varphi\in G$ has a lift in ${\cal G}_G(Q)$ which leaves $B$ invariant. }\\

In this case we obtain a short exact sequence
$$\{1\}\to {\cal G}^B(Q)\to {\cal G}_G^B(Q)\textmap{\pg_B} G\to \{1\},
$$
where  ${\cal G}^B(Q)$  (${\cal G}_G^B(Q)$) is the stabilizer of $B$ in the gauge group ${\cal G}(Q)$ (respectively in the group ${\cal G}_G(Q)$).
  If moreover we can find a lift $\jg:\Gamma \to {\cal G}_G^B(Q)$ of   $\iota_\Gamma: \Gamma\to G$, we will obtain the diagram
  \begin{equation}
\begin{diagram}[s=7mm]
\{1\}&\rTo  &{\cal G}^B(Q) &\rTo  &{\cal G}_G^B(Q) &\rTo^{\pg_B} & G&\rTo & \{1\}	\\
&&&&&\luTo_{\jg}&\uInto_{\iota_\Gamma}  && \\
&&&&&& \Gamma &&
\end{diagram},
\end{equation}
and the quotient bundle $p:P=Q/\Gamma\to N/\Gamma=M$ will comme with an induced connection, which will be denoted by $A$.  The triple $(g, P\textmap{p} M, A)$ will be called the $\Gamma$-quotient of $(h, Q\textmap{q} N,B)$   associated with the lift $\jg: \Gamma\to {\cal G}_G^B(Q)$ of $\iota_\Gamma$. \\

  Condition (C$_G$) has an important gauge theoretical interpretation (see \cite{BiTe} for details}): Let ${\cal B}(Q):={\cal A}(Q)/{\cal G}(Q)$  be the moduli space of all connections on $Q$, where ${\cal A}(Q)$ denotes the  space of connections on $Q$.  A connection $B'$ on a principal $K$-bundle $Q'\simeq Q$ yields a well defined  element    $[B']\in {\cal B}(Q)$ in the following way: we chose a bundle isomorphism $\psi: Q\to Q'$, and we put $[B]:={\cal G}(Q)\cdot \psi^*(B')$. This gauge class will be independent of $\psi$.  On the other hand, since $G$ is connected, we have $\varphi^*(Q)\simeq Q$ for any $\varphi\in G$. Therefore, for any $B\in {\cal A}(Q)$ and any $\varphi\in G$, the pull-back connection $\varphi^*(B)\in {\cal A}(\varphi^*(Q))$ defines a gauge class $\varphi^*[B]  \in {\cal B}(Q)$.   In other words, we obtain a well-defined action of $G$ on the moduli space ${\cal B}(Q)$.
\begin{re}\label{GInv} A connection $B\in {\cal A}(Q)$ satisfies condition ($C_G$) if and only if the gauge class $[B]\in {\cal B}(Q)$ is $G$-invariant.
\end{re}

\begin{re}  If $G$ acts transitively on $N$, then
 
 \begin{itemize}
 \item The pair $(Q,B)$ is homogeneous with respect to the Lie group ${\cal G}_G^B(Q)$. This Lie group is an extension of $G$ by the compact group ${\cal G}^B(Q)$ (which is isomorphic with a closed subgroup of $K$).  
  \item The quotient triple $(g, P\textmap{p} M,A)$ is locally homogeneous.
 \end{itemize}

\end{re} 

The goal of the article is to prove that, under certain (very general) conditions, all locally homogeneous triples can be obtained in this way. More precisely, any locally homogeneous triple on $M$  can be obtained as the quotient of a homogeneous triple on $\tilde M$. In a future article, we will show that, using this result, one obtains an explicit classification theorem  for  locally homogeneous triples with compact structure group $K$  on a given compact manifold $M$.

\newtheorem*{th-main}{Theorem \ref{main}}
\begin{th-main}
Let $M$ be a compact real analytic manifold, and $(g,P\textmap{p} M, A)$ be a real analytic locally homogeneous triple on $M$, where $P\textmap{p} M$ is a $K$-principal bundle with $K$ compact. Let $\pi:\tilde M\to M$ be the universal cover of $M$, $\Gamma$ be the corresponding covering transformation group,  $\tilde g:=\pi^*(g)$, $q:Q:=\pi^*(P)\to \tilde M$, and  $B:=\pi^*(A)$.   Then there exists  
\begin{enumerate}
\item A connected, closed subgroup $G\subset \Iso(\tilde M,\tilde g)$ acting transitively  on $\tilde M$ which leaves  invariant the gauge class $[B]\in {\cal B}(Q)$,
\item A lift $\jg:\Gamma\to {\cal G}_G^B(Q)$ of   $\iota_\Gamma:\Gamma\to G$.
\end{enumerate}

\end{th-main}

\begin{re} In the conditions and with notations of Theorem \ref{main} the natural map $Q\to P$ induces an isomorphism between the  $\Gamma$-quotient of the triple $(\tilde g,  Q\textmap{q} \tilde M, B)$ and  $(g,P\textmap{p} M, A)$. \end{re}
Therefore, using the notations introduced above, we have: 
\newtheorem*{main-coro-intro}{Corollary \ref{main-coro}} 
\begin{main-coro-intro} Let $M$ be a compact real analytic manifold, and $K$ be a compact Lie group.  
Then any real analytic locally homogeneous triple $(g,P\textmap{p} M, A)$ with structure group $K$ on $M$ can be   identified with a $\Gamma$-quotient of the  homogeneous triple $(\tilde g=\pi^*(g),   Q:=\pi^*(P)\textmap{q} \tilde M, B)$ on the universal cover $\tilde M$.	
\end{main-coro-intro}

\begin{re}
An important difficulty in the proof of Theorem \ref{main} is the closedness of the  stabilizer of a class $[B]\in {\cal B}(Q)$ in $\Iso(\tilde M,\tilde g)$. Indeed, since $\tilde M$ is not necessarily compact, the classical gauge theoretical result on the Hausdorff property \cite[p.  130]{DK} of a moduli space  ${\cal B}(Q)$ does not apply.
\end{re}

\begin{re} Corollary \ref{main-coro} shows that, in the condition of Theorem \ref{main}, the classification of real analytic locally homogeneous triples on $M$ reduces to the classification of pairs $(Q,B)$ on the universal cover $\tilde M$, which are homogeneous with respect to a Lie group  ${\cal G}$ fitting into a short exact sequence
$$\{1\}\to L\to {\cal G}\to G\to \{1\},
$$
where $G$ is connected group of isometries of $(\tilde M, \tilde g)$ acting transitively on $\tilde M$ and containing $\Gamma$, and $L$ is a closed subgroup of $K$. The classification of such pairs $(Q,B)$ can be studied using \cite{BiTe}. 
	
\end{re}

\def\an{\mathrm{an}}

\section{Extension of local parallel  sections}

\subsection{The space of parallel sections}
Let $p:P\to M$ be a real analytic principal $K$-bundle over a real analytic manifold $M$.  Let   $\lambda:K\times F\to F$ an analytic action of $K$ on an analytic manifold $F$, and let $E:=P\times_\lambda F$ be the associated bundle with fibre $F$.   Let ${\cal E}\textmap{\mu} M$ be  the projection map of the étale space  of the sheaf of (locally defined) analytic   sections of the bundle $p_E:E\to M$. In other words, a point of ${\cal E}$ is a germ $[s]_x$  where $x\in X$,  and $s\in \Gamma^\an(U,E)$ is an analytic section  defined on  an open neighbourhood $U$ of $x$ in $M$. Two pairs 
$$(s:U\to E,x), \ (s':U'\to E,x')$$
define the same germ (hence one has $[s]_x=[s']_x$)  if $x=x'$ and there exists $U_0\subset U\cap U'$ such that $\resto{s}{U_0}=\resto{s'}{U_0}$. The projection map $\mu$ is given by $[s]_x\mapsto x$. The space ${\cal E}$  has a natural structure of a real analytic  manifold, which is Hausdorff, but does not have countable basis. The topology of ${\cal E}$ can be easily described as follows: Any section $s\in  \Gamma^\an(U,E)$ defines a section $\tilde s:U\to {\cal E}$ of ${\cal E}\textmap{\mu} M$ given by
$$\tilde s (x')=[s:U\to E, x'].
$$
The sets  of the form $\tilde s(U)$ (where $U$ is open in $M$ and $s\in  \Gamma^\an(U,E)$) give a basis for the topology of ${\cal E}$. Note also that, for any open set $U\subset M$, the restriction  $\resto{\mu}{\tilde s(U)}:\tilde s(U)\to U$ is a real analytic diffeomorphism.

Let now   $A$ be a real analytic connection on $P$, and let   $\Gamma^A\subset T_E$ be the associated connection on the associated bundle $E$. Denote by ${\cal E}^A\subset {\cal E}$  the open submanifold of ${\cal E}$ whose points are  germs of $\Gamma^A$-parallel  sections, and by $\mu^A:\cal E^A\to M$ the restriction of $\mu$ to ${\cal E}^A$.  With these notations we state
\begin{thry}\label{CoveringTh}
Suppose that $M$ is connected.  If ${\cal E}^A$ is non-empty, then the map $\mu^A:\cal E^A\to M$   is a covering map.
\end{thry}

\begin{proof}
We will prove that $\mu^A$ satisfies condition (1) in Lemma \ref{coverings}	below. Let $x_0\in M$, and let $h:U_0\to B(0,r)\subset\R^n$ be local chart of $M$ around $x_0$	such that $h(x_0)=0$. We will show that $\mu^A$ maps diffeomorphically any connected component of $(\mu^A)^{-1}(U_0)$ onto $U_0$. 
Using Lemma \ref{s0} below, we obtain, for every point 
$\sigma=[s]_{x_0}\in (\mu^A)^{-1}(x_0)$
  a $\Gamma^A$-parallel section $s_\sigma:U_0\to A$ defining the germ $\sigma$. We claim that the connected components of $(\mu^A)^{-1}(U_0)$ are precisely the sets $\tilde s_\sigma(U_0)$, which are obviously mapped diffeomorphically onto $U_0$ via $\mu^A$. We will first prove that\\ \\
{\it Claim:}   The family $(\tilde s_\sigma(U_0))_{\sigma\in (\mu^A)^{-1}(x_0)}$ is a partition of $(\mu^A)^{-1}(U_0)$.
\\ \\
Indeed, we obviously have  $\tilde s_\sigma(U_0)\subset (\mu^A)^{-1}(U_0)$ for any $\sigma\in (\mu^A)^{-1}(x_0)$, hence 
$$\union_{\sigma\in (\mu^A)^{-1}(x_0)} \tilde s_\sigma(U_0)\subset (\mu^A)^{-1}(U_0).$$
The opposite inclusion is obtained as follows: let $\nu\in (\mu^A)^{-1}(U_0)$, therefore there exists $x\in U_0$ and a $\Gamma^A$-parallel section $s$ defined around $x$ such that $\nu=[s]_x$. Note now that the pair $(x,U_0)$ also satisfies the assumption of Lemma \ref{s0}, because there obviously exists a chart $h_x:U_0\to B(x,r)$ such that $h_x(x)=0$. Therefore the germ $\nu=[s]_x$ extends to a  $\Gamma^A$-parallel section $s_\nu:U_0\to E$. Note  that $[s_\nu]_{x_0}\in (\mu^A)^{-1}(x_0)$, and that the sections $s_\nu$, $s_{[s_\nu]_{x_0}}$ coincide, because they are both parallel and their germ at $x_0$ coincide. This proves that $\nu\in \tilde s_{[s_\nu]_{x_0}}(U_0)$, which proves the inclusion 
$$(\mu^A)^{-1}(U_0)\subset \union_{\sigma\in (\mu^A)^{-1}(x_0)} \tilde s_\sigma(U_0).$$
In order to complete the proof of the claim, it remains to note that the sets $\tilde s_\sigma(U_0)$ are pairwise disjoint. This follows using unique continuation for parallel sections.\\ 

Note now that any set $\tilde s_\sigma(U_0)$ is open in $(\mu^A)^{-1}(U_0)$. Using the claim we see that any such set is also closed in $(\mu^A)^{-1}(U_0)$ (as complement of an open set). Since any such set is obviously connected, the theorem follows by Lemma \ref{coverings}.

\end{proof}

\begin{lm}\label{coverings}
Let $M$, $N$ be differentiable manifolds, and $f:N\to M$ be	a locally diffeomorphic map  satisfying the following properties: %
\begin{enumerate}
\item 	any point $x_0\in M$ has an open neighborhood $U_0$ such that, for any connected component $\tilde U_0$ of $f^{-1}(U_0)$, the map $\tilde U_0 \to U_0$ induced by $f$ is a diffeomorphism,
\item $N$ is non-empty and $M$ is connected.
\end{enumerate}
Then $f$ is a covering projection.
 \end{lm}
 
Usually in the definition of a covering projection, condition (1) and the surjectivity of $f$ are required.  In our case, since $N$ is non-empty and $M$ is connected, the first condition implies the surjectivity of $f$.

\begin{lm}\label{s0}
Let $x_0\in M$, let $[s]_{x_0}\in (\mu^A)^{-1}(x_0)$, where $s:U\to E$ is $\Gamma^A$-parallel, and let $h:U_0\to B(0,r)\subset\R^n$ be an analytic local chart of $M$ around $x_0$	such that $h(x_0)=0$. Then $[s]_x$ extends to a $\Gamma^A$-parallel section $s_0\in\Gamma^{\an}(U_0,E)$, i.e. one has $e=[s_0]_{x_0}$,  for a  $\Gamma^A$-parallel section $s_0\in\Gamma^{\an}(U_0,E)$.
\end{lm}

\begin{proof}
Let $s:U\to E$    be a $\Gamma^A$-parallel section defining the germ $e$, where $U$ is an open neighborhood of $x_0$ in $M$. Let $\varepsilon\in (0,r)$ be sufficiently small such that $U_\varepsilon:=h^{-1}(B(0,\varepsilon))\subset U$. We will construct a real analytic section  $s_0\in\Gamma^{\an}(U_0,E)$ whose restriction to $U_{\varepsilon}$ coincides with $\resto{s}{U_{\varepsilon}}$.  Using \cite[Lemma 2, p. 253]{KN} it  follows  that $s_0$ is $\Gamma^A$-parallel, which completes the proof.

The construction of $s_0$ uses parallel transport along ``radial" curves in $U_0$. More precisely, for any  $u\in U_0$  define the path $\gamma_u:[0,1]\to U_0$ by
$$\gamma_u(t):=h^{-1}(t h(u)), $$
and note that $\gamma_u$ is a path in $U_0$ joining $x_0$ to $u$.	Put $e_0:=s(x_0)\in E_{x_0}$. For any $u\in U_0$ let $\tilde\gamma_u:[0,1]\to E$ be the $\Gamma^A$-horizontal lift of $\gamma_u:[0,1]\to M$ with initial condition $\tilde\gamma_u(0)=e_0$ (see section \ref{ParTran}). Therefore, $\tilde\gamma_u(0)=e_0$, and for every $t\in[0,1]$, one has 
$$\frac{d}{dt}\tilde\gamma_u(t)\in  \Gamma^A_{\tilde\gamma_u(t)}, \ p_E\circ\tilde\gamma_u(t)=\gamma_u(t). $$

Using the analycity of the map $(t,u)\mapsto \tilde\gamma_u(t)$, and a standard theorem on the analycity  of  solutions of  ordinary differential equations with respect to parameters, we see that the map  $u\mapsto s_0(u):= \tilde\gamma_u(1)$ is analytic. On the other hand one has 
$$(p_E\circ s_0)(u)=p_E(\tilde\gamma_u(1))=\gamma_u(1)=u.$$
It remains to prove that $\resto{s_0}{U_\varepsilon}=\resto{s}{U_\varepsilon}$. For this, it suffices to note that for any $u\in U_\varepsilon$ the paths $s_0\circ \gamma_u$, $s \circ \gamma_u$ are both $\Gamma^A$-horizontal lifts of $\gamma_u$ with the same initial condition $e_0$.

\end{proof}

 \subsection{The extension theorem for local parallel sections}

In this section we prove an extension theorem for locally defined parallel sections in analytic associated bundles. This result can be obtained using analytic continuation along paths. This method is used for instance in \cite{KN} for the problem  of extending affine mappings and isometric immersions. We will use a different method, which is base on Theorem \ref{CoveringTh} proved in the previous section.

\begin{thry}\label{ExtSections}
Let $p:P\to M$ be a real analytic principal $K$-bundle over a real analytic manifold $M$, $\lambda:K\times F\to F$ a real  analytic action of $K$ on a real analytic manifold $F$, and $E:=P\times_\lambda F$ the associated bundle with fibre $F$. Let $A$ be an analytic connection on $P$ and  $\Gamma^A\subset T_E$ the associated connection on $E$.

If $M$ is simply connected, then any $\Gamma^A$-parallel section $s:U\to E$ defined on a connected, non-empty open set $U\subset M$ admits a unique  $\Gamma^A$-parallel extension on $M$.\end{thry}

\begin{proof}
 The $\Gamma^A$-parallel section $s$ defines    a section $\tilde s\in\Gamma(U,{\cal E}^A)$. Let ${\cal C}\subset {\cal E}^A$ be  connected component of ${\cal E}^A$ which contains the image  $\tilde s(U)$ of $\tilde s$. By Theorem \ref{CoveringTh} we know that $\mu^A:{\cal E}^A\to M$ is a covering projection, hence  the restriction $\resto{\mu^A}{\cal C}: {\cal C}\to M$ of $\mu^A$ to ${\cal C}$ will also have this property. 
  
But $M$  is simply connected, hence  this  restriction   will be an analytic diffeomorphism.  The inverse map  $\left(\resto{\mu^A}{\cal C}\right)^{-1}:M\to \cal C\subset {\cal E}^A$ will define a parallel extension of $s$. \end{proof}

\section{Extension of  bundle isomorphisms} \label{ExtSect}

We will prove first an extension theorem for locally defined, $\id$-covering bundle isomorphisms intertwining two global connections:

\begin{thry}\label{ExtIso}
Let $p:P\to M$, $p':P'\to M$  be  real analytic principal $K$-bundles over a real analytic manifold $M$, and $A$, $A'$ be   analytic connections on $P$, $P'$ respectively. 
Let $U\subset M$ be a nonempty, connected open set, and $\Phi: P_U\to P'_U$ be an $\id_U$-covering analytic isomorphism such that $\Phi^*(A'_U)=A_U$. 

If $M$ is simply connected, then $\Phi$ has a unique $\id$-covering analytic extension $\tilde \Phi:P\to P'$ such that $\Phi^*(A')=A$.
	\end{thry}

\begin{proof}
For a point $x\in M$ let  $I(P,P')_x$  be the set of all isomorphisms $\psi:P_x\to P'_x$ of  right $K$-spaces. Fixing a pair $(y,y')\in P_x\times P'_x$, we obtain a diffeomorphism   $I(P,P')_x\simeq K$. The union
 $$I(P,P'):=\union_{x\in M} I(P,P')_x$$
 has a natural manifold structure, and the obvious projection $I(P,P')\to M$ is a locally trivial fibre bundle   over $M$ with fiber $K$. More precisely one can identify $I(P,P')$ (as a bundle over $M$) with the associated bundle 
 $$(P\times_M P')\times_\tau K,
 $$
 where $P\times_M P'$ is regarded as a principal $(K\times K)$-bundle over $M$, and 
 $$\tau: (K\times K)\times K\to K$$
  is the action  of $K\times K$ on $K$ defined by 
 $$\tau((k_1,k_2),k):=k_2kk_1^{-1}.$$
 The pair of connections $(A,A')$ defines a connection $A\times A'$ on the product principal bundle $P\times_M P'$.  For a pair $(y,y')\in P\times_M  P'$ the horizontal space $(A\times A')_{(y,y')}$ is just
 $$\{(v,v')\in A_y\times A'_{y'}|\ p_*(v)=p'_*(v')\}.
 $$
  The data of an $\id_U$-covering  morphism $\Phi:P_U\to P'_U$ is equivalent to the data of a section $s^\Phi\in\Gamma(U,I(P,P'))$ (see Proposition \ref{I(P,P')} in the Appendix).  Moreover, by the same proposition, one can prove that $\Phi^*(A')=A$ if and only if the section  $s^\Phi$ is $\Gamma^{A\times A'}$-parallel. With this remark the theorem follows from Theorem \ref{ExtSections}.
\end{proof}

We can treat now the general case of a  locally defined bundle  morphism covering  a globally defined map between the base manifolds: 

\begin{thry}\label{ExtIso-phi}
Let $p:P\to M$, $p':P'\to M'$  be  real analytic principal $K$-bundles, and $A$, $A'$ be   analytic connections on $P$, $P'$ respectively. Let $\varphi:M\to M'$ be a real analytic map,  
 $U\subset M$ be a non-empty, connected open set, and $\Phi: P_U\to P'$ be a  $\varphi$-covering analytic  morphism such that $\Phi^*(A')=A_U$. 

If $M$ is simply connected, then $\Phi$ has a unique $\varphi$-covering analytic extension $\tilde \Phi:P\to P'$ for which $\Phi^*(A')=A$.
	\end{thry}
	
\begin{proof}
By Remark \ref{IdCov} in section \ref{PullBack}, the data of a  $\varphi$-covering analytic bundle  morphism $\phi:P_U\to P'$ for which $\Phi^*(A')=A_U$  is equivalent to the data of an $\id_U$-covering analytic bundle isomorphism $\Phi_0:P_U\to\varphi^*(P')_U$ for which $\Phi_0^*((\varphi^*A')_U)=A_U$. Here we denoted by $\varphi^*(P')$  the bundle on $M$ obtained as the pull-back of  $P'$ via $\varphi:M\to M'$. In other words one has
$$\varphi^*(P'):=M\times_{M'} P'
$$
regarded as a principal $K$-bundle on $M$ via the projection on the first factor. Similarly $\varphi^*(A')$ stands for the pull-back connection of $A'$ via $\varphi$. The connection form of this pull-back connection is the pull-back of the connection form of $A'$ via the second projection $M\times_{(\varphi,p')} P'\to P'$.

By Theorem \ref{ExtIso} it follows that there exists an $\id_M$-covering analytic bundle isomorphism $\tilde\Phi_0:P\to\varphi^*(P')$, which extends  $\Phi_0$ such that   
$$(\tilde\Phi_0)^*(\varphi^*(A'))=A.$$
 But the data of such an isomorphism is equivalent to the data of a $\varphi$-covering analytic bundle morphism $\Phi:P\to P'$ such that $\Phi^*(A')=A$.
\end{proof}
	
\section{Locally homogeneous triples on compact manifolds}

Using our results we can prove now the theorem stated in the introduction:  
\begin{thry}\label{main}
Let $M$ be a compact real analytic manifold, and $(g,P\textmap{p} M, A)$ be a real analytic locally homogeneous triple on $M$, where $P\textmap{p} M$ is a $K$-principal bundle with $K$ compact.  Let $\pi:\tilde M\to M$ be the universal cover of $M$, $\Gamma$ be the corresponding covering transformation group,  $\tilde g:=\pi^*(g)$, $q:Q:=\pi^*(P)\to \tilde M$, and  $B:=\pi^*(A)$.   Then there exists  
\begin{enumerate}
\item A connected, closed subgroup $G\subset \Iso(\tilde M,\tilde g)$ acting transitively  on $\tilde M$ which leaves  invariant the gauge class $[B]\in {\cal B}(Q)$,
\item A lift $\jg:\Gamma\to {\cal G}_G^B(Q)$ of   $\iota_\Gamma:\Gamma\to G$.
\end{enumerate}

\end{thry}

\begin{proof} Let $H\subset \Iso(\tilde M,\tilde g)$ be the subgroup defined by
$$H:=\big\{\psi\in \Iso(\tilde M,\tilde g)|\ \exists \,\Psi:Q\to Q \hbox{ $\psi$-covering bundle isom., }\Psi^*(B)=B\big\}.
$$
Using the fact that $K$ is compact, it follows by Lemma \ref{PropClosed} below, that $H$ is a closed subgroup of the Lie group $\Iso(\tilde M,\tilde g)$. Note that Lemma \ref{PropClosed} applies because the action of the Lie group $\Iso(\tilde M,\tilde g)$ on $\tilde M$ is smooth.\\
\\
{\it Step 1. } $H$ acts transitively on $\tilde M$. \\

 We have to show  that for any two points $\tilde x$, $\tilde x'\in \tilde M$ there exists $\psi\in H$ such that $\psi(\tilde x)=\tilde x'$. Put $x:=\pi(\tilde x)$, $x':=\pi(\tilde x')$. Since 	$(g,P\textmap{p} M, A)$ is locally homogeneous, there exists open neighbourhoods $U$, $U'$ of $x$ and $x'$, an isometry $\varphi:U\to U'$ with $\varphi(x)=x'$,  and a $\varphi$-covering bundle isomorphism $\Phi:P_U\to P_{U'}$ such that $\Phi^*(A_{U'})=A_U$. We can assume of course that $U$ and $U'$ are simply connected.  Let $\tilde U$ ($\tilde U'$) be the connected  component of $\pi^{-1}(U)$ (respectively $\pi^{-1}(U')$) which contains $\tilde x$ (respectively $\tilde x'$).  The map $\pi$ induces diffeomorphisms $\tilde U\simeq U$, $\tilde U'\simeq U'$, and $\pi$-covering bundle isomorphisms $Q_{\tilde U}\simeq P_U$,  $Q_{\tilde U'}\simeq P_{U'}$. Using these identifications we obtain: 
\begin{enumerate}
\item an isometry $\tilde \varphi: \tilde U\to \tilde U'$ with $\tilde \varphi(\tilde x)=\tilde x'$,
\item a  $\tilde \varphi$-covering bundle isomorphism $\tilde \Phi:Q_{\tilde U}\to Q_{\tilde U'}$ with the property 
$$\tilde \Phi^*(B_{\tilde U'})= 	B_{\tilde U}.$$
\end{enumerate}

Since $\tilde M$ is simply connected  and complete, we obtain a global isometry $\psi: \tilde M\to \tilde M$ extending $\tilde\varphi$ (see \cite{Si}, \cite[Theorem 6.3]{KN}). Applying Theorem \ref{ExtIso-phi} we obtain a $\psi$-covering bundle isomorphism $\Psi:Q\to Q$ such that  $\Psi^*(B)=B$. Therefore one has $\psi\in H$. Since $\psi(\tilde x)=\tilde x'$, Step 1 is completed. 
\\ \\
{\it Step 2.} Put $G:=H_0$ (the connected component of $\id_{\tilde M}$ in $H$). We claim that $G$ still acts transitively on $\tilde M$. \\

For every class $\xi\in  H/G$,   choose a  representative $\psi_\xi \in \xi$. In other words one has
$$\xi= G \psi_\xi\  \forall \xi\in   H/G, 
$$
hence $H=\union_{\xi\in H/G} G \psi_\xi$. Fix $\tilde x_0\in \tilde M$. Since $H$ acts transitively on $\tilde M$, we have $\tilde M= H \tilde x_0 $. Therefore
\begin{equation}\label{Dec} \tilde M=\union_{\xi\in H/G} G \psi_\xi(\tilde x_0).
\end{equation}

But $G$ is a closed subgroup of  $\Iso(\tilde M,\tilde g)$, hence its action on $\tilde M$ is proper \cite[Theorem 4]{Ra}.  Therefore the $G$-orbits are embedded closed submanifolds of $\tilde M$. If (by reductio ad absurdum) the $G$-action on $\tilde M$ were not transitive, all these orbits would be submanifolds of dimension strictly smaller than $\dim(\tilde M)$.  But the quotient set $H/G$ is at most countable, hence formula   (\ref{Dec}) would lead to a contradiction. 
\end{proof}

Using the notation introduced in Definition \ref{Hom}  (section \ref{PullBack} in the Appendix) we state: 
 
\begin{lm}\label{PropClosed}
Let $M$ be a differentiable manifold, $K$ be a compact Lie group, $p:P\to M$, $p':P'\to M$ be   principal $K$-bundles on $M$, and  $A\in{\cal A}(P)$, $A'\in{\cal A}(P')$ be  connections on $P$, $P'$ respectively. Let $\alpha:L\times M\to M$ be a smooth action of a Lie group	$L$ on   $M$. For $l\in L$ denote by $\varphi_l:M\to M$ the corresponding diffeomorphism. The subspace
$$L_{AA'}:=\{l\in L|\ \exists \Phi\in\Hom_{\varphi_l}(P,P')\hbox{ such that }\Phi^*(A')=A\} 
$$
is closed in $L$. In the special case $P=P'$, $A=A'$, the obtained subset $L_{AA}$ is a Lie subgroup of $L$.
\end{lm}

\begin{proof}
Let $(l_n)_{n\in \N}$ be   sequence in $L_{AA'}$  converging 	to an element $l_{\infty}\in L$. We will prove that $l_\infty\in L_{AA'}$. 

Let $V\subset T_e L$ be a sufficiently small convex neighbourhood of $0$ in the tangent space $\lg=T_e L$ of $L$ at its unit element $e$, such that the exponential map $\exp:T_eL\to L$ induces a diffeomorphism $E:V\to U$. The map $\eta:[0,1]\times U \to U$ defined by 
$$\eta(t,l)=E(tE^{-1}(l)))
$$
is a smooth homotopy joining the constant map $e$ on $U$ to the identity map $\id_U$. Moreover one has
\begin{equation}\label{lim1}
\lim_{l\to e} \eta(t,l)=e	
\end{equation}
uniformly on $[0,1]$. 
For any $l\in U$ and $x\in M$ we obtain a smooth path $\gamma_x^l:[0,1]\to M$ given by
$$\gamma_x^l(t)=\varphi_{\eta(t,l)}(x)
$$
which joins $x$ to $\varphi_l(x)$. For $l\in U$ denote by $\gamma^l:[0,1]\times M\to M$   the map $(t,x)\mapsto \gamma^l_x(t)$.   Denoting by $\mathrm{p}_M:[0,1]\times M\to M$ the projection on the $M$-factor, and using (\ref{lim1}) we get
\begin{equation}
\label{lim2}\lim_{l\to e} \gamma^l=\mathrm{p}_M
\end{equation}
 in the weak topology ${\cal C}^\infty_w([0,1]\times M,M)$.

Fix any connection $B'$ on  $P'$. For any $l\in U$ we get a  $\varphi_l$-covering automorphism $\Psi_l\in\Hom_{\varphi_l}(P',P')$ defined by
$$\Psi_l(y')=\{\widetilde{\gamma_{p'(y')}^l}\}^{B'}_{y'}
$$
where $\{\widetilde{\gamma_{p'(y')}^l}\}^{B'}_{y'}$ is the $B'$-horizontal lift of $\gamma_{p'(y')}^l$ with initial $\{\widetilde{\gamma_{p'(y')}^l}\}^{B'}_{y'}(0)=y'$.  Using  (\ref{lim2}) we get
\begin{equation}
\label{lim3}\lim_{l\to e} \Psi_l=\id_{P'}
\end{equation} 
in the weak topology ${\cal C}^\infty_w(P',P')$.  Put 
$$\lambda_n:=l_\infty l_n^{-1}.
$$
Since $\lim_{n\to\infty} l_n=l_\infty$ we may suppose that $\lambda_n\in U$ for any $n$. Let $\Phi_n\in \Hom_{\varphi_{l_n}}(P,P')$ such that $\Phi_n^*(A')=A$, and note that 
$$\Sigma_n:=\Psi_{\lambda_n}\circ \Phi_n  
$$
is a $\varphi_{l_\infty}$-covering bundle morphism with the property
$$(\Sigma_n)_*(A)=(\Psi_{\lambda_n})_*(A').
$$
Using (\ref{lim3}) and $\lim_{n\to \infty} \lambda_n=e$, we obtain 
$$\lim_{n\to \infty} (\Sigma_n)_*(A)=A'
$$
in the Fréchet affine space ${\cal A}(P')$.  The claim follows now from Lemma \ref{newlemma} below. 

\end{proof}

\begin{lm}\label{newlemma}
Let $K$ be a compact Lie group, $p:P\to M$, $p':P'\to M'$ be   principal $K$-bundles on $M$ and $M'$, and  $A\in{\cal A}(P)$, $A'\in{\cal A}(P')$ be  connections on $P$, $P'$ respectively. Let $\varphi: M\to M'$ be a smooth map and $(\Phi_n)_n$ be a sequence in $\Hom_\varphi(P,P')$  such that   $\lim_{n\to\infty}\Phi_n^*(A')=A$. Then there exists a subsequence $(\Phi_{n_k})_k$ of $(\Phi_n)_n$ which converges (in the weak ${\cal C}^\infty$-topology) to a morphism $\Phi_\infty\in \Hom_\varphi(P,P')$ satisfying $\Phi_\infty^*(A')=A$.
\end{lm}
\begin{proof}
It suffices to prove the claim in the special case   when $M'=M$, $\varphi=\id_M$. Indeed, using Remark \ref{IdCov} (see section \ref{PullBack}) we obtain morphisms $\Phi_n^0\in\Hom_{\id_M}(P,\varphi^*(P'))$ such that
$$\lim_{n\to\infty} (\Phi_n^0)^*(\varphi^*(A'))=A
$$ 
Since the bijection given by Remark \ref{IdCov} is a homeomorphism with respect to the weak  ${\cal C}^\infty$-topology it suffices to prove the claim for the sequence of $\id_M$-covering morphisms  $(\Phi_n^0)_n$.

From now we suppose that $M'=M$, $\varphi=\id_M$, $\Phi_n\in\Hom_{\id_M}(P,P')$.  We will use ideas form gauge theory \cite[section 2.3.7]{DK}. Note that 
$$F_n:=\Phi_1^{-1}\circ\Phi_n\in\Hom_{\id_M}(P,P)={\cal G}(P)=\Gamma(P\times_\iota K),$$
where ${\cal G}(P)$ is the gauge group of $P$ and $\iota:K\to\Aut(K)$ is the interior morphism defined by $\iota(k)(u)=kuk^{-1}$. Therefore we have a sequence $(F_n)_n$ of gauge transformations of $P$ with the property
\begin{equation}\label{GG}
\lim_{n\to \infty} (F_n)_*(A)= \Phi_1^*(A')=:A_1.	
\end{equation}
Fix an embedding $\rho:K\to \mathrm{O}(N)$, and let  $E:=P\times_\rho\R^N$ be the Euclidian associated bundle. Any gauge transformation $F\in {\cal G}(P)$ can be identified with a section in the endomorphism bundle $\End(E)$, and via this identification one has
$$F_*(A)=A-(\nabla^A F)F^{-1},
$$ 
where $\nabla^A$ is the linear connection induced by $A$ on $\End(E)$. Therefore, putting $B_n:=(F_n)_*(A)$, we get
\begin{equation}\label{GGG}\nabla^A F_n=(A-B_n)F_n 
\end{equation}
where, by (\ref{GG}), 
\begin{equation}
\label{GGGG}\lim_{n\to \infty} B_n=A_1
\end{equation}
 in the  Fréchet ${\cal C}^\infty$-topology of the affine  space ${\cal A}(P)$. Since $K$ is compact, $F_n$ is uniformly bounded on $M$. On  the other hand by (\ref{GGGG}) the sequence $(A-B_n)_n$ is bounded in  the Fréchet space $\Gamma(\Lambda^1_M\otimes\ad(P))\subset\Gamma(\Lambda^1_M\otimes\End(E))$. Using (\ref{GGG}) and a standard bootstrapping procedure we see that all partial derivatives of $F_n$ (with respect to local coordinates in $M$ and trivializations of $P$) are uniformly bounded on any compact subset of $M$.  Using a well-known combination of the Arzela-Ascoli theorem and the diagonal argument we obtain a subsequence $(F_{n_k})_k$ of $(F_n)$ which converges in $\Gamma(\End(E))$ (with respect to its Fréchet ${\cal C}^\infty$-topology) to a section $F_\infty\in \Gamma(\End(E))$. Since $\rho(K)$ is closed in $\mathrm{gl}(N,\R)$, it follows that ${\cal G}(P)$ is closed in $\Gamma(\End(E))$, hence $F_\infty\in {\cal G}(P)$. Replacing $n$ by $n_k$ and taking  $k\to \infty$ in (\ref{GGG}) we get $\nabla^A F_\infty= (A-A_1) F_\infty$, i.e. $A_1=(F_\infty)_*(A)$. It suffices to put $\Phi_{n_k}=\Phi_1\circ F_{n_k}$, $\Phi_\infty:= \Phi_1\circ F_\infty$.
\end{proof}
As explained in section \ref{LocHom-intro}, using Theorem \ref{main} shows that:
\begin{co}\label{main-coro} Let $M$ be a compact real analytic manifold, and $K$ be a compact Lie group.  
Then any real analytic locally homogeneous triple $(g,P\textmap{p} M, A)$ with structure group $K$ on $M$ can be   identified with a $\Gamma$-quotient of the  homogeneous triple $(\tilde g=\pi^*(g),   Q:=\pi^*(P)\textmap{q} \tilde M, B)$ on the universal cover $\tilde M$.	
\end{co}

\section{Appendix}

\subsection{Parallel transport in associated bundles} \label{ParTran}
Let $M$ be a differentiable $n$-manifold, $K$ be a Lie group, and $p:P\to M$ be a principal $K$-bundle on $M$, endowed with a connection $A$. Let $\alpha:[t_0,t_1]\to M$ be a smooth curve in $M$. By a well known, fundamental theorem in Differential Geometry it follows that any point $u_0\in P_{\alpha(t_0)}:=p^{-1}(\alpha(t_0))$ there is a unique horizontal lift $\alpha_{u_0}:[t_0,t_1]\to P$ such that $\alpha_{u_0}(t_0)=u_0$ \cite[Proposition 3.1]{KN}.

We also have an existence and unicity horizontal lift theorem  for  associated bundles. More precisely, let $p:P\to M$ be a principal $K$-bundle endowed with a connection $A$,    $\lambda:K\times F\to F$ be a smooth action of $K$ on a differentiable manifold $F$, and let $E:=P\times_\lambda F$ be the associated bundle with fibre $F$. The connection $A$ is a horizontal  rank $n$-distribution on $P$, which defines in the obvious way a rank $n$-distribution on $P\times F$, whose projection on $E=(P\times F)/K$ is a well defined horizontal rank $n$-distribution on $E$. This distribution will be called the connection on $E$ induced by $A$, and will be denoted by $\Gamma^A$.  With this definitions we have:

\begin{pr}
Let $p:P\to M$ be a principal $K$-bundle endowed with a connection $A$,    $\lambda:K\times F\to F$ be a smooth action of $K$ on a manifold $F$, and let $E:=P\times_\lambda F$ be the associated bundle with fibre $F$.  Let $\alpha:[t_0,t_1]\to M$ be a smooth curve in $M$. For any point $e_0\in E_{\alpha(t_0)}$ there is a unique $\Gamma^A$-horizontal lift $\alpha_{e_0}:[t_0,t_1]\to E$ such that $\alpha_{e_0}(t_0)=e_0$. 
\end{pr}

A proof of this results is sketched in \cite[p. 88]{KN}.  In this section we give a detailed proof for completeness.

\begin{proof}
The Lie group $K$ acts freely on $P\times F$ from the right by 
$$ ((u,y),k)\mapsto (uk,k^{-1}y).$$
By definition the associated bundle $E:=P\times_\lambda F$ is the quotient manifold $(P\times F)/K$, and the projection $q:P\times F\to E$ is principal $K$-bundle, which comes with a natural connection $B$ given by
$$B_{(u,y)}:=A_u\times T_y F\ \forall (u,y)\in P\times F.
$$
The bundle projection map $p:P\to M$ induces a locally trivial submersion $p_E:E\to M$, such that the following diagram is commutative:
$$
\begin{diagram}[s=7mm]
P\times F&\rTo{q} & E\\
\dTo^{p_1} & &\dTo_{p_E} \\
P &\rTo{p} & M\ . \\
\end{diagram}  
$$
In other words the quotient map $q:P\times F\to E$ is   fiber preserving map over $M$. For a pair $(u,y)\in P\times F$  put  $[u,y]:=q(u,y)$. With this notation the map   $p_E:E\to M$ is given by $p_E([u,y]):=p(u)$.   By the definition of $\Gamma^A$, for a point $e:=[u,y]\in E$ we have  
$$  \Gamma^A_{e}:=q_{*(u,y)}(A_u\times\{0_y\})$$
where $\{0_y\}$ is the zero subspace of $T_yF$. Put $x_0:=\alpha(t_0)$, and let  
 $e_0=[u_0,y_0]$ be a point in the fiber $E_{x_0}$.  It follows $u_0\in p^{-1}(x_0)$.

 By the classical horizontal lift theorem \cite[Proposition 3.1]{KN}, there exists a  $A$-horizontal lift $\alpha_{u_0}:[t_0,t_1]\to P$. Therefore  for every $t\in[t_0,t_1]$ we have $\alpha_{u_0}'(t)\in A_{\alpha_{u_0}(t)}$. It suffices to note that the curve  $\alpha_{e_0}:[t_0,t_1]\to E$   defined by 
 $$\alpha_{e_0}(t):=[\alpha_{u_0}(t),y_0]$$
is a $\Gamma^A$-horizontal lift of $\alpha$ satisfying the initial condition $\alpha_{e_0}(t_0)=e_0$.

To prove unicity, let $\beta: [t_0,t_1]\to E$ be a $\Gamma^A$-horizontal lift of $\alpha$ with $\beta(t_0)=e_0$, and let 
$$\beta_{(u_0,y_0)}=(\beta^P_{(u_0,y_0)},\beta^F_{(u_0,y_0)}):[t_0,t_1]\to P\times F$$
 be a $B$-horizontal lift of $\beta$ satisfying the initial condition $\beta_{(u_0,y_0)}(t_0)=(u_0,y_0)$. Since $\beta_{(u_0,y_0)}$ is a $B$-horizontal curve, it follows that 
\begin{equation}\label{eq1}\beta'_{(u_0,y_0)}(t)\in B_{\beta_{(u_0,y_0)}(t)}\ \forall t\in [t_0,t_1],
	\end{equation}
and, since $\beta_{(u_0,y_0)}(t_0)=(u_0,y_0)$ is a lift of a $\Gamma^A$-horizontal curve, it follows that 
\begin{equation}\label{eq2}\beta'_{(u_0,y_0)}(t)\in \big\{(q_*)^{-1}(\Gamma^A)\big\}_{\beta_{(u_0,y_0)}(t)}\ \forall t\in [t_0,t_1].
\end{equation}
But it is easy to show that  the intersection $B\cap (q_*)^{-1}(\Gamma_A)$ of the distributions $B$, $(q_*)^{-1}(\Gamma^A)\subset T_{P\times F}$ coincides with the distribution $\tilde A\subset T_{P\times F}$ given by
$$\tilde A_{(u,y)}:=A_u\times \{0_y\}. 
$$
Therefore, by (\ref{eq1}), (\ref{eq2}) it follows that $\beta^P_{(u_0,y_0)}$ is $A$-horizontal, and $\beta^F_{(u_0,y_0)}$ is constant.
Using the unicity part of \cite[Proposition 3.1]{KN}, we obtain  $\beta^P_{(u_0,y_0)}=\alpha_{u_0}$, hence $\beta_{(u_0,y_0)}(t)=(\alpha_{u_0}(t),y_0)$, hence $\beta_{(u_0,y_0)}=\alpha_{e_0}$.

\end{proof}

\begin{re}
A similar theorem holds for general fiber bundles endowed with a connection, but one has to assume that the fiber is compact \cite{Mo}.
\end{re}

\subsection{Pull back bundles and pull back connections}  \label{PullBack}

Let $M$ and $M'$ be smooth manifolds, $K$ a Lie group, and  $\varphi:M\to M'$ be   a smooth map.
Let $p':P'\to M'$ be a principal $K$-bundle on $M'$. Recall that the pull-back bundle $\varphi^*(P')$ is defined by
$$\varphi^*(P'):=M\times_{M'} P'=(\varphi\times p')^{-1}(\Delta_{M'}),
$$
where $\varphi\times p':M\times P'\to M'\times M'$ denotes the product map, and $\Delta_{M'}\subset M'\times M'$ is the diagonal submanifold. One can check that $\varphi\times p'$ is transversal to  $\Delta_{M'}$, hence $M\times_{M'} P'$ is a submanifold of $M\times P'$, whose   tangent space  at a point $(x,y')$ is  
$$T_{(x,y')}(M\times_{M'} P')=\{(u,w)\in TM\times TP' \ |\ \varphi_*(u)=p'_*(w)\}.
$$
Note that set theoretically one has
$$\varphi^*(P')=\bigsqcup_{x\in M}P'_{\varphi(x)}.
$$
The  projections   $p_1:\varphi^*(P')\to M'$, $p_2:\varphi^*(P')\to P'$ fit in the commutative diagram 
$$
\begin{diagram}[s=7mm]
\varphi^*(P') & \rTo^{p_2}  & P'\\	
\dTo^{p_1} &  &\dTo{p'} \\
M & \rTo^{\varphi }  & M'
\end{diagram}
$$
in which $p_1$ is a $K$-principal bundle on $M$ with respect to the right $K$-action induced from $P'$.\\

Let  $A'$  be a connection on   $P'$ and $\omega^{A'}\in A^1(P,\kg)$ its connection form. By definition the pull-back connection $\varphi^*(A')$ is the connection on $\varphi^*(P')$ defined by the connection form $p_2^*(\omega^{A'})$. With this definition we have for any pair $(x,y')\in \varphi^*(P')$:
$$\varphi^*(A')_{(x,y')}=\{(u,w)\in T_xM\times T_{y'}(P)|\ \varphi_{*x}(u)=p'_{*y'}(w),\ w\in A'_{y'}\}  $$
$$=\{(u,w)\in T_xM\times A'_{y'}|\ \varphi_{*x}(u)=p'_{*y'}(w)\}.
$$

\begin{dt} \label{Hom} Let $M$, $M'$ be smooth manifolds,  $K$ a Lie group, and  $\varphi:M\to M'$ be  a smooth map.  Let $P\textmap{p} M$, $P'\textmap{p'} M'$ be   principal $K$-bundles 	on $M$, $M'$  respectively. We denote by $\Hom_\varphi(P,P')$ the set of $\varphi$-covering bundle  morphisms $P\to P'$  which are compatible with the group morphism $\id_K$  \cite[section I.5]{KN}.  
	
\end{dt}
For a $\varphi$-covering bundle morphism $\Phi\in \Hom_\varphi(P,P')$ and a connection $A'$ on $P'$ one defines the pull-back connection $\Phi^*(A')$ using the  connection form $\Phi^*(\omega^{A'})$ \cite[Proposition 6.2]{KN}.
 
\begin{re} \label{IdCov} With the notations above the following holds:
\begin{enumerate}
\item 	The map $\Hom_\varphi(P,P')\to  \Hom_{\id_M}(P,\varphi^*(P'))$ given by $\Phi\mapsto\Phi_0$, where
$$\Phi_0(y):=(p(y),\Phi(y))
$$
is bijective.
\item For any connection $A'$ on $P'$ and bundle morphism  $\Phi\in \Hom_\varphi(P,P')$  one has
$$\Phi^*(A')=\Phi_0^*(\varphi^*(A')).
$$
\end{enumerate}
\end{re}

 \subsection{Bundle isomorphisms compatible with a pair of connections}

Let $P$, $P'$ be $K$-principal bundles over a manifold $M$, and let $(A,A')\in {\cal A}(P)\times{\cal A}(P')$ be a pair of connections. The goal of this section is the proof of Proposition \ref{I(P,P')}, which shows that the data of an $\id$-covering bundle morphism $\Phi:P\to P'$ with the  property $\Phi^*(A')=A$ is equivalent to the data of an $(A,A')$-parallel section in a bundle $I(P,P')$ associated with the fibre product $P\times_M P'$ and the action $\tau:(K\times K)\times K\to K$ given by $((k_1,k_2),k)\mapsto k_2 k k_1^{-1}$ (see section \ref{ExtSect}).  The first part of the following proposition is well-known. The second part can be checked easily.

\begin{pr}\label{EquivSec}
Let $p:P\to M$ be a principal $K$-bundle,  $\alpha:K\times F\to F$ be a smooth left action of $K$ on a manifold $F$, and $E:=P\times_KF$ be  the associated fiber bundle with fiber $F$. Denote by ${\cal C}^K(P,F)$ the space of smooth $K$-equivariant maps $P\to F$:
$$ {\cal C}^K(P,F):=\{\sigma:P\to F|\ \sigma(y.k)=k^{-1}.\sigma(y)|\ \forall y\in P,\ \forall k\in K \}. $$
 Then
\begin{enumerate}
\item The map $\cal F:{\cal C}^K(P,F)\to \Gamma(M,E)$ given by 
$$\sigma\mapsto s_\sigma, \ s_\sigma(x):=[y,\sigma(y)], \hbox{ where }y\in P_x,
$$
is bijective. 
\item   Let $\sigma\in {\cal C}^K(P,F)$, $A$ be a connection on $P$, and  $\Gamma^A$ be  the induced connection on $E$. The following conditions are equivalent:
\begin{enumerate}[(i)]
\item  The  section $s_\sigma$ is $\Gamma^A$-parallel.  
\item  The restriction of $\sigma_{*y}$ to the horizontal distribution of $A$ vanishes.
\end{enumerate}

\end{enumerate}
\end{pr}
\begin{proof}
(1) This result is well known. We mention only that the inverse of ${\cal F}$ is the map $s\mapsto \sigma^s$ where, for a section $s\in\Gamma(E)$, the equivariant map $\sigma^s$ is defined  by the identity
\begin{equation}\label{ssigma}
s(p(y))=[y,\sigma^s(y)], \forall y\in P.	
\end{equation}
\\
(2) Let $q:P\times F\to E$ be  the quotient map. By (\ref{ssigma}) we know that 
\begin{equation}\label{ssigmaq}s_\sigma(x)=[y,\sigma(y)]=q(y,\sigma(y)).
\end{equation}

Let $v\in T_xM$,  $y\in p^{-1}(x)$, and let $w$ be a lift of $v$ in $T_yP$. Using (\ref{ssigmaq}) we obtain easily:
\begin{equation}
 s_{*x}(v)=q_{*(y,\sigma(y))}(w,\sigma_{*y}(w)).
 \end{equation}
Recall that, by the definition of $\Gamma^A$, we have $q_{*(y,\sigma(y))}(A_y\times \{0\})=\Gamma^A_{s(x)}$. Therefore, if $\resto {\sigma_{*} } {A}=0$,
then, choosing $w$ to be the horizontal lift of $v$ at $y$, we get $s_{*x}(v)\in\Gamma^A_{s(x)}$. Conversely, supposing that for every $v\in T_xM$ we have $s_{*x}(v)\in\Gamma^A_{s(x)}$, we show that 
$\resto {\sigma_{*} } {A}=0$. Let $y\in P$, $w\in A_y$, and $v=p_*(w)$. Using again (\ref{ssigmaq}) we see that    

$$s_{*}(v)=q_{*}(w, \sigma_{*}(w))\in q_{*}(A_y\times \{0\}).$$
Therefore, there exists $u\in A_y$ such that 
$$(w-u,\sigma_{*}(w))\in\ker (q_{*(y,\sigma(y))}).
$$
The projection on the first factor induces an isomorphism $\ker (q_{*(y,\sigma(y))})\to V_y$, where $V_y\subset T_yP$ denotes the vertical tangent space at $y$. Taking into account that  $w-u\in A_y$, we get $w-u=0$, and $\sigma_*(w)=0$. 
\end{proof}
\begin{pr}\label{I(P,P')}
Let $p:P\to M$, $p':P'\to M$ be  $K$-principal bundles over $M$, and let $I(P,P')$ be the associated bundle $(P\times_M P')\times_\tau K$. There exist a natural bijection $S:\Hom_{\id}(P,P')\to \Gamma(M,I(P,P'))$ between the space of   $\id$-covering $K$-bundle isomorphisms and the space of sections   $\Gamma(M,I(P,P'))$ with the following property: For any pair of connections $(A,A')\in {\cal A}(P)\times {\cal A}(P')$ the  following conditions are equivalent:
\begin{enumerate}[(i)]
\item  $\Phi^*(A')=A$.
\item  $S(\Phi)$ is $\Gamma^{A\times A'}$-parallel.
\end{enumerate}

\end{pr}
\begin{proof}
By Proposition \ref{EquivSec} the space of sections   $\Gamma(M,I(P,P'))$ is naturally identified with the space ${\cal C}^{K\times K}(P\times_M P',K)$  of  $(K\times K)$-equivariant maps   $P\times_M P'\to K$. Using this identification  we will define a natural bijection  
$$S:\Hom_{\id}(P,P')\to {\cal C}^{K\times K}(P\times_M P',K).$$
Let $\Phi \in\Hom_\id(P,P')$, and  $(y,y')\in P\times_M P'$. Since $\Phi(y)$ and $y'$ are in the same fiber, there exists a unique element $\sigma^\Phi(y,y')\in K$ such that $\Phi(y)=y'.\sigma^\Phi(y,y')$. It is easy to see that the map $P\times_M P'\ni (y,y')\mapsto \sigma^\Phi(y,y')$ is $(K\times K)$-equivariant, hence it gives an element $\sigma^\Phi\in {\cal C}^{K\times K}(P\times_M P',K)$. Therefore, by definition, we get the identity 
\begin{equation}\label{idt}
\Phi(y)=y'.\sigma^\Phi(y,y')	.
\end{equation}
Conversely, for an element $\sigma\in {\cal C}^{K\times K}(P\times_M P',K)$, it is easy to check that the right hand side of (\ref{idt}) depends only on $y$ and defines an $\id$-covering bundle morphism $P\to P'$. Our bijection $S$ is $\Phi\mapsto\sigma^\Phi$.\\

We prove now that, for any pair $(A,A')\in {\cal A}(P)\times {\cal A}(P')$, the conditions (i), (ii) are equivalent.

(2) Let $\Phi\in\Hom_\id(P,P')$.  Let $\lambda:P'\times K\to P'$ be  the right action of $K$ on $P'$. For a pair $(y',k)\in P'\times K$ denote by  $\lambda_{y'}:K\to P'$, $\lambda_k:P'\to P'$ the corresponding  maps obtained from $\lambda$ by fixing an argument. Using (\ref{idt}) we obtain for a pair  $(w,w')\in T(P\times_M P')$
\begin{equation}\label{Leibniz}
\Phi_{*y}(w)=(\lambda_{y'})_{*}\sigma^\Phi_{*}(w,w') +(\lambda_{\sigma^\Phi(y,y')})_{*}(w').	
\end{equation}

Let ${\cal H}\subset T_{P\times_M P'}$ be the $(A,A')$-horizontal distribution. Recall that one has  ${\cal H}_{(y,y')}=A_y\times_{T_xM} A'_y$.   Suppose that   $\resto{\sigma^\Phi_{*}}{\cal H}=0$, let $w\in A_y$ and let $w'$ be the $A'$-horizontal lift of $p_*(w)$.   Using (\ref{Leibniz}) we obtain
 $$\Phi_{*y}(w)=(\lambda_{\sigma^\Phi(y,y')})_{*}(w')\in A'_{y'.\sigma^\Phi(y,y')}=A'_{\Phi(y)}$$
 Therefore $\Phi_{*y}(A_y)=A'_{y'}$. Since this holds for any $(y,y')\in P\times_M P'$ we get $\Phi^*(A')=A$,   as claimed. Conversely, suppose $\Phi^*(A')=A$.  Then for any $(w,w')\in\cal H_{(y,y')}$ we have
  $$(\lambda_{y'})_{*}\sigma^\Phi_{*}(w,w') =-(\lambda_{\sigma^\Phi(y,y')})_{*}(w')+\Phi_{*y}(w)\in A'_{\Phi(y)},$$
  where the left hand side is vertical, and the right hand side is $A'$-horizontal. This implies $(\lambda_{y'})_{*}\sigma^\Phi_{*}(w,w')=0$, hence  $\sigma^\Phi_{*}(w,w')=0$.  
\end{proof}

\end{document}